\documentclass[10pt,letterpaper]{article}
\pdfoutput=1
\usepackage{graphicx}
\usepackage{setspace}
\usepackage{amssymb,amsmath}
\usepackage{calc}
\usepackage{verbatim}
\usepackage{epsfig}
\usepackage{graphicx}
\usepackage{graphics}
\usepackage{multicol}
\usepackage{enumitem}
\usepackage{enumerate}
\usepackage{cite}

\newtheorem{theorem}{Theorem}
\newtheorem{lemma}[theorem]{Lemma}

\newtheorem{observation}[theorem]{Observation}

\newcommand{\noin}{\noindent}
\newenvironment{proof}{{\noin \bf Proof}: }{\newqed \medskip}

\newcommand{\newqed}{\null\hfill$\Box$\par\medskip}

\newcommand{\ipoly}[2]{I(#1,#2)}

\title{Optimal Graphs for Independence and $k$-Independence Polynomials}

\author{J.I. Brown$^1$, D. Cox $^2$}

\begin{document}

\maketitle

\begin{small}
\begin{center}
	$^1$ Department of Mathematics and Statistics, Dalhousie University,\\ Halifax, Canada B3H 4R2, Jason.Brown@dal.ca\\
	$^2$ Department of Mathematics, Mount Saint Vincent University,\\ Halifax, Canada B3M 2J6, danielle.cox@msvu.ca\\
	October 4, 2017
\end{center}
\end{small}

%%%%%%%%%%%%%%%%%%%%%%%%%%%%%%%%%%%%%%%%%%%%%%%%%%%%%%%%%%%%%%%%%%%%%%%%%%%%%%%%%%%%%%%%%%%%%

\begin{abstract}
The {\em independence polynomial} $I(G,x)$ of a finite graph $G$ is the generating function for the sequence of the number of independent sets of each cardinality. 
We investigate whether, given a fixed number of vertices and edges, there exists {\em optimally-least} ({\em optimally-greatest}) graphs, that are least (respectively, greatest) for {\em all} non-negative $x$. 
Moreover, we broaden our scope to $k$-independence polynomials, which are generating functions for the $k$-clique-free subsets of vertices. 
For $k \geq 3$, the results can be quite different from the $k = 2$ (i.e. independence) case.
\end{abstract}

%%%%%%%%%%%%%%%%%%%%%%%%%%%%%%%%
\section{Introduction}

Given a property of subsets of the vertex or edge set -- such as independent, complete, dominating for vertices and matching for edges -- one is often interested in maximizing or minimizing the size of the set in a given graph $G$. However, one can get a much more nuanced study of the subsets by studying the number of such sets of each cardinality in $G$, and in this guise, one often encapsulates the sequence by forming a generating function. Independence, clique, dominating and matching polynomials have all arisen and been studied in this setting.

In all cases, the generating polynomial $f(G,x)$ is naturally a function on the domain $[0,\infty)$. 
If the number of vertices $n$ and edges $m$ are fixed, one can ask whether there exists an extremal graph. Let $\mathcal{S}_{n,m}$ denote the set of simple graphs of order $n$ ($n$ vertices) and size $m$ ($m$ edges). Let $H\in \mathcal{S}_{n,m}$. $H$ is {\em optimally-greatest} (or {\em optimally-least}) if $f(H,x) \geq f(G,x)$ ($f(H,x) \leq f(G,x)$, respectively) for all graphs $G\in \mathcal{S}_{n,m}$ and {\em all} $x \geq 0$ (for any particular value of $x \geq 0$, of course, there is such a graph $H$, as the number of graphs of order $n$ and size $m$ is finite, but we are interested in {\em uniformly} optimal graphs). 

Such questions (related to simple substitutions of generating polynomials) have attracted considerable attention in the areas of network reliability \cite{ath,suffel,boesch91,browncox,gross,myrvold} and chromatic polynomials \cite{sakaloglu,chromatic}. 

Here we consider optimality for independence polynomials
\[ I(G,x) = \sum_{j} i_{j}x^{j},\]
and a broad generalization, for fixed $k \geq 2$, to $k$-independence polynomials
\[ I_k(G,x) = \sum_{j} i_{k,j}x^{j},\] 
where $i_{k,j}$ is the number of subsets of the vertex set of size $j$ that induce a $k$-clique-free subgraph (that is, the induced subgraph contains no complete subgraph of order $k$). Clearly independence polynomials are precisely the $2$-independence polynomials, while the $3$-independence polynomials are generating functions for the numbers of triangle-free induced subgraphs.

In this paper, we look at the optimality of independence polynomials and more generally $k$-independence polynomials.
In the case of the former, optimally-greatest graphs always exist for independence polynomials, but we do not know whether optimally-least graphs necessarily exist for independence polynomials as well (although we will prove for some $n$ and $m$ they do). In contrast, we shall show that for $k\geq 3$, optimality behaves quite differently, in that for some $n$ and $m$, optimally-least and optimally-greatest graphs for $k$-independence polynomials do not exist.

\section{Optimally for Independence Polynomials}

\subsection{Optimally-Greatest Graphs for Independence Polynomials}
%We begin by looking at the existence of optimally-greatest graphs for the $k$-independence polynomial on $[0,\infty)$. We note that while it is clear that there is an optimally-greatest graph for the $k$-independence polynomial for $x \geq 0$ sufficiently large (or any {\em particular} $x \in (0,\infty)$), as there are only finitely many such graphs, what is not clear is whether there is such a graph for {\em all} $x \geq 0$.

Our first result shows that for independence polynomials, optimally-greatest graphs indeed do always exist.

\begin{theorem}
For all $n \geq 1$ and all $m \in \{0,\ldots,{n \choose 2}\}$, an optimally-greatest graph always exists. 
\end{theorem}
\begin{proof}
A key observation is that a sufficient condition for $H$ to be optimally-greatest (or optimally-least) for independence polynomials is that $i(H,x) = \sum i_{j}(H) x^{j}$ is {\em coefficient-wise greatest}, that is, for all other graphs $G$ of the same order and size, the generating polynomials $f(G,x) = \sum i_{j}(G) x^{j}$ satisfies $i_{j}(H) \geq i_{j}(G)$ for all $j$ (respectively, $i_{j}(H) \leq i_{j}(G)$). (The coefficient-wise condition is not, in general, necessary for optimality as, for example, $5x^{2}+x+5 \geq x^{2}+4x+1$ for $x \geq 0$, but clearly $5x^{2}+x+5$ is not coefficient-wise greater than or equal to $x^{2}+4x+1$.)

Consider the following graph construction. For a given $n$ and $m$, take a fixed linear order $\preceq$ of the vertices, $v_{n} \preceq v_{n-1} \preceq \cdots \preceq v_{1}$, and select the $m$ largest edges in lexicographic order. We will denote this graph as $G_{n,m,\preceq}$ (it is dependent on the linear order, but of course all such graphs are isomorphic). It was shown in \cite{cutler} that $G_{n,m,\preceq}$ is the graph of order $n$ and size $m$ with the most number of independent sets of size $j$, for {\em all} $j \geq 0$. 
%Their proof consisted of optimizing the F-vector of the independence complex, given $F_0=1$, $F_1=n$ and $F_2={n\choose 2}-m$. 
It follows that the independence polynomial for $G_{n,m,\preceq}$ is coefficient-wise optimally-greatest for independence polynomials, for all graphs of order $n$ and size $m$.
\newqed 
\end{proof}

Moreover, we can compute the independence polynomial of this optimally-greatest graph. To do so, we need the following well-known recursion to compute the independence polynomial of a graph.
Let $G$ be a graph with $v\in V(G)$. Then
\begin{eqnarray}
I(G,x)=xI(G-N[v],x)+I(G-v,x) \label{ipoly}
\end{eqnarray}
where $N[v]$ is the closed neighbourhood of $v$ and $G-v$ is $G$ with $v$ removed.
For graph $G_{n,m,\preceq}$ with $m < {{n} \choose {2}}$, write $m=(n-\ell)+(n-\ell+1)+\ldots (n-1)+j$ with $j \in \{0,\ldots,n-\ell-2\}$. Since the edges are added in lexicographic order, there will be $k$ vertices that have degree $n-1$ and one vertex of degree $j+\ell$. The remaining $n-\ell-1$ vertices form an independent set. Let $v$ be the vertex of degree $j+\ell$. 
Using vertex, $v$ and recursion (\ref{ipoly}), we obtain our result,
\begin{eqnarray*}
 I(G_{n,m,\preceq},x)& = &I(G_{n,m,\preceq}-v,x)+xI(G_{n,m,\preceq}-N[v],x)\\
  & = & I(K_{\ell} \cup \overline{K_{n-\ell-1}},x) + I(K_{\ell} \cup \overline{K_{n-\ell-j-1}},x)\\
   & = & (1+\ell x)(1+x)^{n-\ell-1} + x(1+\ell x)(1+x)^{n-\ell-j-1}\\
   & = & (1+\ell x)(1+x)^{n-\ell-j-1} \left( (1+x)^{j} + x\right).
 \end{eqnarray*}
Of course, for $m = {n \choose 2}$, $G_{n,m,\preceq} = K_{n}$ and so $I(G_{n,m,\preceq},x) = 1+nx$.

\vspace{0.25in}

It is instructive that this result can also be proved purely algebraically, and we devote the remainder of this section to do so. (We refer the reader \cite{brownbook}
for an introduction to complexes and their connection to commutative algebra.) 
The {\em independence complex} of graph $G$ of order $n$ and size $m$, $\Delta_{2}(G)$, has as its faces the independent sets of $G$ (these are obviously closed under containment, and hence a complex). The {\em $f$-vector} of $\Delta_{2}(G)$ is $(1,n,{{n} \choose {2}} - m,f_{3},\ldots,f_{\beta})$, where $f_{i}$ is the number of faces of cardinality $i$ in the complex (and $\beta$ is the independence number of $G$, which is the same as the dimension of the complex). We will show that we can maximize the independence polynomial on $[0,\infty)$ by maximizing ({\em simultaneously}, for some graph $G$) all of the $f_{i}$'s, via an excursion into commutative algebra. 

We begin with some definitions. Fix a field ${\mathbf k}$. Let $A$ be a {\em ${\mathbf k}$-graded algebra}, that is, $A$ is a commutative ring containing ${\mathbf k}$ as a subring, that can be written as a vector space direct sum $\displaystyle{A = \bigoplus_{d\geq 0} A_{d}}$, over ${\mathbf k}$, with the property that $A_{i}A_{j} \subseteq A_{i+j}$ for all $i$ and $j$ (we call elements in some $A_{i}$ {\em homogeneous}, and $A_{i}$ is called the {\em $d$-th graded component} of $A$). The graded ${\mathbf k}$-algebra $A$ is {\em standard} if it is generated (as a ring) by a finite set of elements in $A_{1}$. Our prototypical example of a standard ${\mathbf k}$-graded algebra is the polynomial ring ${\mathbf k}[x_{1},x_{2},\ldots,x_{n}]$ in variables $x_{1},x_{2},\ldots,x_{n}$. Note that for any standard graded ${\mathbf k}$-algebra $\displaystyle{A = \bigoplus_{d\geq 0} A_{d}}$ that is a quotient of a polynomial ring by a homogenous ideal, a ${\mathbf k}$-basis for $A_{d}$ is simply the monomials in $A_{d}$.

The {\em Stanley-Reisner complex} of an ideal $I$ of a standard graded ${\mathbf k}$-algebra $A$ (generated by $x_{1},x_{2},\ldots,x_{n}$ in $A_{1}$) is the (simplicial) complex whose faces are the square-free monomials in $x_{1},x_{2},\ldots,x_{n}$ not in $I$ (the properties of being an ideal ensures that this set is closed under containment). 
Let $I$ be an ideal of ${\mathbf k}$-algebra $A$; $I$ is {\em homogeneous} if it is generated by homogeneous elements of $A$. We write $\displaystyle{I = \bigoplus_{d\geq 0} I_{d}}$, where $I_{d} = A_{d} \cap I$ is the {\em $d$-th graded component of $I$} (it is a ${\mathbf k}$-subspace of $I$). For a homogeneous ideal $I$ of $A$, a square-free monomial $M$ of degree $d$ in $x_{1},\ldots,x_{n}$ belongs to exactly one of $I_{d}$ and the Stanley-Reisner complex of $I$ (where we identify a face of the complex with the product of its elements). As the total number of monomial of $Q$ of degree $d$ is fixed, we see that maximizing the number of faces of size $d$ in the Stanley-Reisner complex of $I$ corresponds to minimizing the number of monomials of degree $d$ in $I$. 

The {\em Hilbert function} of the homogeneous ideal $I$ is the function $H_{I}:{\mathbb N} \rightarrow {\mathbb N}$, where $H_{I}(d) = \mbox{dim}_{\mathbf k}I_d$. We call $I$ {\em Gotzmann} (see \cite{hoefelpaper}) if for all other homogeneous ideals $J$ of $A$ and all $d \geq 0$, if $H_{I}(d)=H_{J}(d)$ then $H_{I}(d+1)\leq H_{J}(d+1)$. 
%If each graded component of $I$ is Gotzmann, then we call $I$ a {\em Gotzmann ideal} \citep{hoefelpaper}.
For an ideal of $Q$ which is Gotzmann, its Hilbert function is smallest for each value of $d \in {\mathbb N}$. 
%If a monomial is in the $d$-th graded component of $Q$, but not in the $d$-th graded component of $I$, then it must be in the Stanley-Reisner complex.  That is, given an ideal $I$ of $Q$, any monomial in $Q$ is either in $I$ or in the Stanley-Reisner complex of $I$.  Gotzmann ideals are ideals with the smallest Hilbert function growth, meaning if an ideal, $(I_d)$ of a ring $R$ is Gotzmann, their Hilbert function is smaller than the other homogeneous ideals of $R$ with the same dimension in degree $d$. So, if $(I_d)$ is Gotzmann, for any other homogeneous ideal $J\subseteq R$, if $H((I_d),d)=H(J,d)$ then $H((I_d),k)\leq H(J,k)$ for all $k\geq d$ .

%for each $d$-th graded component, and thus for each $d$ the coefficients $f_d$ in the $f$-vector of the Stanley-Riesner complex are at least as largest as any other homogeneous ideal of $Q$.  

% We call In this case $A/I$ is also a graded ${\mathbf k}$-algebra, with $\displaystyle{A/I = \bigoplus_{d\geq 0} I_{d}}$, over ${\mathbf k}$, where $I_{d} = A_{d} \cap I$ is the $d$-th graded component of $I$ . This vector space contains $0$ and all the homogeneous polynomials of degree $d$ which are in the ideal, $I$. 

We will now focus in on a standard graded ${\mathbf k}$-algebra related to independence in graphs. Fix $n \geq 1$. The  {\em Kruskal-Katona ring}, $Q = {\mathbf k}[x_{1},\ldots,x_{n}]/\langle x_{1}^{2},\ldots,x_{n}^{2}\rangle$, is generated by the square-free monomials; it is clearly a standard graded $k$-algebra. Let $G$ be a graph on vertices $\{x_1,x_2,\ldots x_n\}$. The {\em edge ideal} $I_{G}$ is the ideal of $Q$ generated by $\{x_ix_j \mid x_ix_j \in E(G) \}$. If a (square-free) monomial of $Q$ is not in $I_G$ then that set of vertices cannot contain an edge in $G$, so it is an independent set (and vice versa). This means that the Stanely-Reisner complex of our edge ideal $I_{G}$ in the Kruskal-Katona ring is precisely the independence complex $\Delta_{2}(G)$ of our graph $G$. If we can show that the edge ideal $I_{G}$ is Gotzmann in $Q$, then this means that for each $d \geq 0$, $I_{G}$ contains the fewest monomials of degree $d$ for all $d$ compared to {\em any} other such edge ideal. Hence by a previous observation, the $f$-vector of the independence complex of $G$ will have the largest entries component-wise compared to any other graph of order $n$ and size $m$. Thus our graph will have an independence polynomial that is optimally-greatest.

Let $I$ be an ideal in a standard graded ${\mathbf k}$-algebra $A$, generated by $x_{1},\ldots,x_{n} \in A_{1}$. Then $I$ is a { \em lexicographic ideal} if for any monomials $u$ and $v$ in $x_{1},\ldots,x_{n}$, whenever $v \in I$ and $u$ is lexicographically bigger than $v$, we have $u \in I$ as well.
It is known that lexicographic ideals are Gotzmann in the Kruskal-Katona rings \cite{macaulaylex}. As the edges of our family of graphs $G_{n,m,\preceq}$ are added in lexicographic order, it follows that the edge ideal of $I_G$ in $Q$ is lexicographic, and hence Gotzmann. Therefore the $f$-vector is maximized, given $f_0, f_1, f_2$, that is, given, $n$ and $m$, and so $G_{n,m,\preceq}$ is optimally-greatest.

\vspace{0.25in}

\subsection{Optimally-Least Graphs for Independence Polynomials}

We now turn our attention to the existence of optimally-least graphs for the independence polynomial, and we find here that we can only prove the existence of optimally-least graphs for certain values of $m = m(n)$ (and we do not know if there are values of $n$ and $m$ for which optimally-least graphs do not exist). 

We begin with dense graphs. It is obvious that for a graph $G$ of order $n$ and size $m$ (which has  ${n \choose 2}-m$ many independent sets of cardinality $2$) that the independence polynomial of such a graph $G$ is, for $x \geq 0$, at least 
\[ 1+nx+\left( {n \choose 2}-m \right)x^2,\]
and by a previous observation, if there is a graph with this independence polynomial, then it is the optimally-least. 

Turan's famous theorem states (see, for example \cite{thebook}) that the maximum number of edges of a graph with no triangles is $\lceil \frac{n}{2} \rceil \lfloor \frac{n}{2} \rfloor = \lfloor \frac{n^{2}}{4} \rfloor$, with equality iff $G$ is a complete bipartite graph with sides of equal or nearly equal cardinality. It follows, by taking complements, that provided 
\[ m \geq {{n} \choose {2}} - \Bigl \lfloor \frac{n^{2}}{4} \Bigr \rfloor \]
then the graph formed by adding any $m - ({{n} \choose {2}} - \lfloor \frac{n^{2}}{4} \rfloor)$ edges to the complement of $K_{ \lceil \frac{n}{2} \rceil, \lfloor \frac{n}{2} \rfloor}$, namely $\overline{K_{ \lceil \frac{n}{2} \rceil, \lfloor \frac{n}{2} \rfloor}} ~=~ K_{\lceil \frac{n}{2} \rceil} \cup K_{\lfloor \frac{n}{2} \rfloor}$ is the optimally-least graph. This gives the following result:

\begin{theorem}
For a given $n\geq 2$ and $m\geq {{n} \choose {2}} - \lfloor \frac{n^{2}}{4} \rfloor$, the graph with the optimally-least independence polynomial is formed from $K_{\lceil \frac{n}{2} \rceil} \cup K_{\lfloor \frac{n}{2} \rfloor}$ by adding in any $m - ({{n} \choose {2}} - \lfloor \frac{n^{2}}{4} \rfloor)$ edges. The independence polynomial of such a graph is 
\[ 1+nx+\left( {n \choose 2}-m \right)x^2.\] 
\newqed
\end{theorem}

We can extend this result by utilizing a result of Lov\'{a}sz and Simonovits \cite{lovsim}, which, answering a conjecture of Erd\"{o}s, showed that for $1 \leq k < n/2$, the ($K_{4}$-free) graph of order $n$ and size $\lfloor n^{2}/4 \rfloor$ with the fewest number of triangles is the graph formed from $K_{ \lceil \frac{n}{2} \rceil, \lfloor \frac{n}{2} \rfloor}$ by adding in edges to a largest cell so that no triangles are formed within the cell. (In such a case, the number of triangles formed in the graph is $k \lfloor n/2 \rfloor$.) As well, a theorem of Fisher and Solow \cite{fishersolow} states that for $a \geq b \geq 1$, the least number of triangles in a $K_{4}$-free graph of order $n=2a+b$ and size $m=2ab+a^{2}$ occurs in $K_{a,a,b}$. By taking complements, we derive:

\begin{theorem}
For a given $n\geq 2$ and $m = {{n} \choose {2}} - \lceil \frac{n}{2} \rceil \lfloor \frac{n}{2} \rfloor - k$, where $1 \leq k \leq \lfloor n/2 \rfloor$, the graph with the optimally-least independence polynomial is formed from $K_{\lceil \frac{n}{2} \rceil} \cup K_{\lfloor \frac{n}{2} \rfloor}$ by deleting any $k$ edges in $K_{\lceil \frac{n}{2} \rceil}$ so that no independent set of size $3$ is formed in that part. The independence polynomial of such a graph is 
\[ 1+nx+\left( {n \choose 2}-m \right)x^2+ k \lfloor n/2 \rfloor x^{3}.\]
As well, for any for $a \geq b \geq 1$, the graph with the optimally-least independence polynomial with order $n=2a+b$ and size $m = a(a-1)+b(b-1)/2$ is $2K_{a} \cup K_{b}$, which has independence polynomial 
\[ 1+nx+\left( {n \choose 2}-m \right)x^2+\left( 2{a \choose 3} + {b \choose 3}\right)x^3.\] \newqed
\end{theorem}

To find sparse families of optimally-least graphs we will look at a graph operation which can be done to decrease the value of the independence polynomial on $[0,\infty)$. 
Let $H_1$ be a graph which consists of an induced subgraph $G_1$, containing an edge $e = vw$, and two other vertices $y$ and $z$ that are isolated, and let $H_2$ be the graph formed from $H_1$ by removing edge $e$ and adding in an edge between $y$ and $z$ (see Figure~\ref{polyinc} -- we set $G_{2} = G_{1} - e$). Clearly $G_1$ and $G_2$ have the same number of vertices and edges.
We will show that this removal of an edge to form a $K_{2}$ can never increase the independence polynomial (on $[0,\infty)$).

\begin{figure}[ht]
\centering
\includegraphics[width=4in]{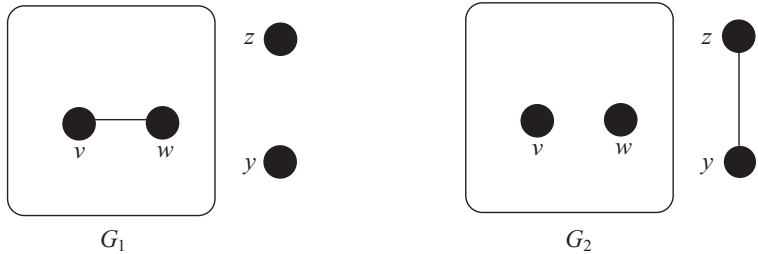}
\caption{Graphs for Lemma~\ref{lelmlem}. $H_1$ is the graph $G_1 \cup \{y,z\}$ (on the left) and $H_2$ is the graph $G_2 \cup \{y,z\}$ (on the right).}
\label{polyinc}
\end{figure}

\begin{lemma} 
\label{lelmlem}
For the graphs in Figure~\ref{polyinc}, we have $I(H_2,x)\leq I(H_1,x)$ on $[0,\infty)$.
\end{lemma}

\begin{proof}
Note that by Equation (\ref{ipoly}), our deletion-contraction formula for independence polynomials, we have
\begin{eqnarray*}
I(H_1,x)& = & (1+2x+x^2)I(G_1,x)\\
    & = & (1+2x+x^2)(I(G_1-v,x)+xI(G_1-N[v],x))\\
    &=& (1+2x+x^2)(I(G_2-v,x)+xI(G_2-N[v]-w,x))
\end{eqnarray*}
 and
\begin{eqnarray*}
I(H_2,x)& = &(1+2x)(I(G_2-v,x)+xI(G_2-N[v],x)).
\end{eqnarray*}
We also find that 
\begin{eqnarray}\label{inequal} I(G_2-N[v],x)\leq (1+x)I(G_2-N[v]-w,x)  \end{eqnarray} since $G_2-N[v]$ is a subgraph of $(G_2-N[v]-w) \cup K_{1}$.
%\[ \{ S: S \mbox{ is independent in } G_{2}-N[v]-w \}  \cup  \{ S \cup \{w\}: S \mbox{ is independent in } G_{2}-N[v]-w \}\]
%contains every independent set of $G_2-N[v]$. 

Consider $F(x)=I(H_1,x)-I(H_2,x)$. Using our expression for $i(H_{1},x)$ and $i(H_{2},x)$ and the inequality (\ref{inequal}), we can see that
\begin{eqnarray*}
F(x) & = & (1+2x)I(G_2-v,x)+x^2I(G_2-v,x)+\\
      && x(1+2x+x^2)\ipoly{G_2-N[v]-w}{x}-(1+2x)\ipoly{G_2-v}{x}\\
       &&-x(1+2x)\ipoly{G_2-N[v]}{x}\\
     &=& x^2\ipoly{G_2-v}{x}+x(1+x)\ipoly{G_2-N[v]-w}{x}\\
      &&+x^2(1+x)\ipoly{G_2-N[v]-w}{x}-x(1+2x)\ipoly{G_2-[v]}{x}\\
     &\geq & x^2\ipoly{G_2-v}{x}+x\ipoly{G_2-N[v]}{x}\\
     &&+x^2\ipoly{G_2-N[v]}{x}-x\ipoly{G_2-[v]}{x}-2x^2\ipoly{G_2-N[v]}{x}\\
     &=& x^2\ipoly{G_2-v}{x}-x^2\ipoly{G_2-N[v]}{x}.
\end{eqnarray*}     

Since $G_2-N[v]$ is a subgraph of $G_2-v$, we have $\ipoly{G_2-v}{x}\geq \ipoly{G_2-N[v]}{x}$. It follows that $F(x)\geq 0$, so $\ipoly{H_1}{x}\geq \ipoly{H_2}{x}$.
\newqed 
\end{proof}

It follows that if $m\leq n/2$, by pulling out $K_2$'s, we derive:
% the graph that is optimally-least for the independence polynomial is the disjoint union of $mK_2$ and $(n-2m)K_1$ (where $\ell G$ is the graph consisting of $\ell$ disjoint copies of $G$). This gives the following result.

\begin{theorem}
\label{lessthanhalf}
For a given $n\geq 2$ and $m \leq \frac{n}{2} $, the optimally-least graph for the independence polynomial for $x \geq 0$ is $mK_2 \cup (n-2m)K_1 $. \newqed
\end{theorem}

\section{Optimality for $k$-Independence Polynomials}

We now look at the optimality of $k$-independence polynomials and find the situation is much different for $k \geq 3$ than for $k = 2$.  We will show that in contrast, to $k=2$,  for all $k \geq 3$, there {\em does not} always exist optimally-greatest nor optimally-least graphs for the $k$-independence polynomial. 

Before we do so, we make the following observation, which shows that the nonexistence of optimal graphs can sometimes be derived by considering only certain coefficients of the polynomials.

\begin{observation}
\label{max}
 Suppose that $G$ and $H$ be graphs on $n$ vertices and $m$ edges and let $k\geq 2$, with
\[ {\rm I}_{k}(G,x)=\sum_{j=0}^ni_j(G)x^{j} \] and
\[ {\rm I}_{k}(H,x)=\sum_{j=0}^ni_j(H)x^{j} \]
Then 
\begin{itemize}
  \item if $i_j(G)=i_j(H)$ for $j<\ell$ but $i_{\ell}(G)>i_{\ell}(H)$, then {\rm I}$_{k}(G,x)>${\rm I}$_{k}(H,x)$ for $x$ arbitrarily small and 
  \item if $i_j(G)=i_j(H)$ for $t>j$ but $i_{t}(G)>i_{t}(H)$, then {\rm I}$_{k}(G,x)>${\rm I}$_{k}(H,x)$ for $x$ arbitrarily large.
\end{itemize}
\end{observation}

%If $I_k(G,x)>I_k(H,x)$ on a domain $D$, we say that $G$ is more optimal than $H$ on $D$, or that $H$ is less optimal than $G$ on $D$. 
Note that $i_{j} = {n \choose j}$ for $j < k$. For a graph $G$, let $r_{G} = r^{k}_{G}$ denote the largest value of $j$ such that there exists an induced subgraph of $G$ of order $j$ that does not contains a $k$-clique, that is, $r_{G}$ is the largest value of $j$ such that $i_j(G) > 0$.

Thus to show that for $k$-independence polynomials ($k \geq 3$) there does not always exist optimally-greatest graphs, we will show that for some $n$ and $m = m(n)$, there is a unique graph $G \in \mathcal{S}_{n,m}$ with the largest value of $r_{G}$ in $\mathcal{S}_{n,m}$, thus optimally-greatest for sufficiently large values of $x$, but there is another graph $H \in \mathcal{S}_{n,m}$ with more $k$-independent sets than $G$, and hence optimally-greatest the $k$-independence polynomial for arbitrarily small values of $x\geq 0$.

\begin{theorem}
For $k \geq 3$ and $l \geq 2$, and any $n > (k-1)l(l-1)$, there does not exist an  optimally-greatest graph for the $k$-independence polynomial of order $n$ and size $m = {n \choose 2}-(k-1){l \choose 2}$.
\end{theorem}
\begin{proof}
We recall an old well known result by Turan (see \cite{bollobas2}, for example) that states that the unique graph $T_{n,k}$ of order $n$ with the maximum number $m_{n,k}$ of edges in a graph of order $n$ without a $k$-clique is the complete $(k-1)$-partite graph with cells of order $\lfloor n/(k-1) \rfloor$ or $\lceil n/(k-1) \rceil$. For fixed $n > (k-1)l(l-1)$, we set $m = {n \choose 2}-(k-1){l \choose 2}$ and consider the class $\mathcal{S}_{n,m}$. We define the graph $G = G(n,m)$ as the join $T_{(k-1)l,m} + K_{n-(k-1)l}$ of the Turan graph $T_{(k-1)l,m}$ with $K_{n-(k-1)l}$, that is, $G$ is formed from the disjoint union of $T_{(k-1)l,k}$ with $K_{n-(k-1)l}$ by adding in all edges between them (equivalently, the complement of $G$, $\overline{G}$, consists of $k-1$ cliques of order $l$, together with $n-(k-1)l$ isolated vertices). A quick calculation shows that $G$ has $m$ edges, and hence belongs to $\mathcal{S}_{n,m}$.

We claim first that among all graphs $F^{\prime}$ in $\mathcal{S}_{n,m}$, $G$ has the largest value of $r_{F^{\prime}}$. Note that $r_{G} = (k-1)l$ as the Turan graph $T_{(k-1)l,k}$ has no $k$-clique. Now if there is a graph $F$ in $\mathcal{S}_{n,m}$ with $r_{F} > r_{G}$, then $F$ has an induced subgraph $S$ on say $s > (k-1)l$ vertices with no $k$-clique, and hence $S$ has at most as many edges as $T_{s,m}$. However, then $\overline{F}$ has at least as many edges as $\overline{T_{s,m}}$, which is strictly more than the number of edges in $\overline{T_{(k-1)l,m}}$ (to see this, think of the complements of Turan graphs being the disjoint unions of cliques, and observe that for $t > (k-1)l$, one can form $\overline{T_{t,m}}$ from $\overline{T_{(k-1)l,m}}$
 by successively adding vertices to the cliques to keep them as nearly equal as possible). Thus $\overline{F}$ would contain more edges than $\overline{G}$, the disjoint union of $\overline{T_{(k-1)l,m}}$ and isolated vertices, and hence $F$ would fewer edges than $G$, a contradiction as both $F$ and $G$ have the same number of vertices and edges. We conclude no such $F$ exists, so $G$ has the maximal $r_F$ value. Moreover, by the argument given, if $G^{\prime}$ were a graph in $\mathcal{S}_{n,m}$ with $r_{G^{\prime}} = r_{G}$, then if $S$ is any induced subgraph of $G^{\prime}$ of size $r_{G}=(k-1)l$ that has no $k$-clique, $\overline{S}$ must have precisely as many edges as $\overline{T_{(k-1)l,k}}$, which is the number of edges of $\overline{G}$. We conclude that, from Turan's Theorem, that $G^\prime$ is isomorphic to $G$, which is therefore the unique graph in $\mathcal{S}_{n,m}$ with the largest $r_F$-value, and hence the unique graph optimally-greatest for the $k$-independence polynomial for $x$ sufficiently large.

So if there is an optimally-greatest graph the $k$-independence polynomial for $\mathcal{S}_{n,m}$, it must be $G$. We note that if a graph  of order $n$ and size $m$ has a minimum number of $k$-cliques, then it has a maximum number of $k$-independent sets of order $k$. We will show now that there is another graph $H \in \mathcal{S}_{n,m}$ with fewer $k$-cliques than $G$, and hence more $k$-independent sets of size $k$, and so $G$ cannot be optimally-greatest the $k$-independence polynomial as $I_{k}(H,x) > I_{k}(G,x)$ for $x> 0 $ sufficiently small. 

Let $H$ be the graph of order $n$ such that $\overline{H}$ is the disjoint union of $(k-1){l \choose 2}$ $K_{2}$'s and isolated vertices (as $n > (k-1)l(l-1) = 2(k-1){l \choose 2}$, we can find such a graph of order $n$). We can think of $\overline{H}$ as splitting the edges of the $l$-cliques of $\overline{G}$ into edges. 
Now it is easy to see that for graphs $G_{1}$ and $G_{2}$ having $k_{1,i}$ and $k_{2,i}$ cliques of order $i$, respectively, then for any positive integer $t$, the number of $k$-cliques of order $t$ in $G_{1}+G_{2}$ is 
\[ \sum_{i+j = t} k_{1,i}k_{2,j}.\]
It follows that it suffices to show that the following graph
$G_{1} = \overline{K_{l}+l(l-2)K_{1}}$
has, for all $i$, at least as many $i$-cliques as the graph $G_{2} = \overline{{l \choose 2}K_{2}}$, and that for $i \geq 3$ the former has strictly more $i$-cliques than the latter. For $i = 1 $ and $i = 2$ they have the same number of $i$-cliques (as they have the same number of vertices and edges). For $i \geq 2$, the number of $i$-cliques in $G_{1}$ is 
\[ {l(l-2) \choose i} + l {l(l-2) \choose {i-1}}\]
while $G_{2}$ has 
\[ {{l \choose 2} \choose i}2^{i}\]
many $i$-cliques.
We set
\[ f_{l,i} = \frac{{l(l-2) \choose i} + l {l(l-2) \choose {i-1}}}{{{l \choose 2} \choose i}2^{i}}.\]
Clearly $f_{l,2} = 1$ as both $G_{1}$ and $G_{2}$ has the same number of edges. We'll show that the sequence is strictly increasing, and hence for $i \geq 3$, always greater than $1$, so that $G_{1}$ has strictly more $i$-cliques than $G_{2}$, concluding the proof.

Now via some straightforward but tedious calculations, we find that for $i \geq 2$,
\[ \frac{f_{l,i+1}}{f_{l,i}} =  \frac{(l(l-2)-i+l(i+1))(l(l-2)-i+1)}{(l(l-1)-2i)(l(l-2)-i+1+li)}.\]
It follows that 
\[ \frac{f_{l,i+1}}{f_{l,i}} > 1 \]
iff 
\begin{eqnarray*} 
(l(l-2)-i+l(i+1))(l(l-2)-i+1) & > & (l(l-1)-2i)(l(l-2)-i+1+li),
\end{eqnarray*} which holds iff 
\[ i(l-1)(i-1) > 0.\]
The latter is true as  $i, l \geq 2$.

Thus we conclude that $f_{l,i} > 1$ for all $i \geq 2$. Thus $H$ has fewer $k$-cliques than $G$, hence more $k$-independent sets of order $k$ and so an optimally-greatest graph the $k$-independence polynomial does not exist.
\newqed 
\end{proof}

\vspace{0.25in}

We turn finally to the issue of optimally-least $k$-independence polynomials.
From Observation~\ref{max}, we have that if an optimally-least graph the $k$-independence polynomial exists, then it must have the least number of $k$-independent sets of order $k$ (or equivalently, for the optimally-least graph the $k$-independence polynomial has the maximum number of $k$-cliques), since a graph with the latter will be optimally-least the $k$-independence polynomial for sufficiently small values of $x\geq 0$. In \cite{bollobas2} it was shown that for a graph on $n$ vertices and $m={d \choose 2}+r$ edges ($0 \leq r < d$) and for $k \geq 3$, the maximum number of cliques of size $k$ is ${d \choose k}+{r \choose k-1}$, and a graph which achieves such bounds consists of a $K_d$, a vertex $x$ with $N(x)\subseteq V(K_d)$, and $n-d-1$ isolated vertices (see Figure~\ref{optkleastnear0}). (This graph is not the unique extremal graph for values of $r<k-1$, since the addition of fewer than $k-1$ edges will not produce another $k$-clique, but for values of $r\geq k-1$, this graph is the unique extremal graph, since to obtain the maximum number of $k$-cliques, the edges will have to be added to the same vertex.)  Such graphs are candidates for optimally-least graphs the $k$-independence polynomial, but we will show that for $k\geq 3$, that optimally-least graphs the $k$-independence polynomial do not always exist. 

\begin{figure}[ht]
\centering
\includegraphics[width=1.8in]{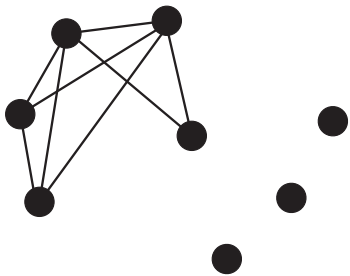}
\caption{A graph that is optimally-least for the $k$-independence polynomial near $0$.}
\label{optkleastnear0}
\end{figure}

\begin{theorem}
For $k\geq 3$, $n \geq {k\choose 2} + k  + 1$ vertices and $m={{k\choose 2}+2\choose 2}-1$ optimally-least graphs the $k$-independence polynomial do not exist.
\end{theorem}

\begin{proof}
Let $k\geq 3$. Let $G=\left( K_{{k\choose 2}+2}-e \right) \cup (n-{k \choose 2}-2)K_1$, for some edge $e$ of $K_{{k\choose 2}+2}$, and let $H=K_{{k\choose 2}+1}\cup K_k\cup (n-{k \choose 2}-1-k)K_1$. We know that $G$ is optimally-least for the $k$-independence polynomial, when $x$ is sufficiently close to $0$, and $H$ is not. We will now show that $H$ is optimally-least for the $k$-independence polynomial for arbitrarily large values of $x$.

In $G$, the largest size of a vertex set that does not contain a $K_k$ is $n-{k \choose 2}+k-2$, taking any of the $n-{k\choose 2}-2$ isolated vertices, end points of edge $e$ and $k-2$ vertices from $K_{{k\choose 2}+2}$. In $H$, the largest size of a vertex set that does not contain a $K_k$ is size $n-{k \choose 2}+k-3$, taking any of the isolated $n-{k \choose 2}-1-k$ vertices and $k-1$ vertices from each complete graph. Thus by Observation~\ref{max}, $H$ is optimally-least the $k$-independence polynomial for arbitrarily large values of $x$, and $G$ is not, so no optimally-least graphs exist for the $k$-independence polynomial for these $n$ and $m$.
\newqed 
\end{proof}

%So, we have shown that for the independence polynomial that optimally-greatest graphs always exist and to contrast, optimally-least graphs do not always exist. For small orders, $n\leq 9$ optimally-least graphs always exist for the independence polynomial. To show whether this is the case for all $n$ is an interesting open problem. We also introduced the $k$-independence polynomial and demonstrated that unlike the independence polynomial, optimally-greatest and optimally-least graphs need not always exist. 

\section{Conclusion}

While we have seen that optimally-greatest graphs exist for independence polynomials, we have only be able to prove the existence of optimally-least polynomials for a restricted collection of $n$ and $m$. Our belief is that such graphs always exist, but there does not seem to be any reasonable family to put forward as extremal.

In terms of optimality for $k$-independence polynomials, we have seen that for all $k \geq 3$, there are infinitely many values of $n$ and $m$ such that optimally-greatest graphs do not exists, and similarly for optimally-least graphs. 

A full characterization of when optimal graphs for the $k$-independence polynomial exist (for $k\geq 3$, and even for optimally-least for $k=2$) remains open.

\vspace{0.25in}
\noindent {\bf Acknowledgements}
\vspace{0.1in}

\noindent J.I. Brown acknowledges support from NSERC (grant application RGPIN 170450-2013). D. Cox acknowledges research support from NSERC (grant application RGPIN 2017-04401) and Mount Saint Vincent University.

\bibliography{k_indp_poly_oct2_arxiv.bbl}
 
 \end{document}